\documentclass[11pt]{article}
\usepackage{mathrsfs}
\usepackage{amsfonts,amsmath,dsfont,amssymb,amsthm,stmaryrd,bbm,color}
\usepackage{hyperref}

\newtheorem{conj}{Conjecture}[section]
\newtheorem{thm}{Theorem}[section]
\newtheorem{rmk}[conj]{Remark}
\newtheorem{lem}[conj]{Lemma}
\newtheorem{prop}[conj]{Proposition}

\newtheorem{defn}[conj]{Definition}

\newcommand{\supp}{\mathrm{supp}}

\newcommand{\R}{\mathbb{R}}

\def\S{{\mathbb S}}

\begin{document}
\title{R\'enyi entropy power inequality and a reverse}
\author{Jiange Li}
\date{\today}
\maketitle

\begin{abstract}
This paper is twofold. In the first part, we present a refinement of the R\'enyi Entropy Power Inequality (EPI) recently obtained in \cite{BM16}. The proof largely follows the approach in \cite{DCT91} of employing Young's convolution inequalities with sharp constants. In the second part, we study the reversibility of the R\'enyi EPI, and confirm a conjecture in \cite{BNT15, MMX16} in two cases. Connections with various $p$-th mean bodies in convex geometry are also explored.
\end{abstract}

\section{Introduction}

Let $X$ be a random vector in $\R^n$ with density $f$ with respect to the Lebesgue measure. Its Shannon entropy is defined as
$$
h(X)=-\int_{\R^n} f(x)\log f(x)dx.
$$
The entropy power of $X$ is defined as
$$
N(X)=e^{2h(X)/n}.
$$
Shannon's Entropy Power Inequality (i.e., EPI) states that for any two independent random vectors $X$ and $Y$ in $\R^n$ such that
the entropies of $X$, $Y$ and $X+Y$ exist,
\begin{align}\label{eq:epi}
N(X+Y)\geq N(X)+N(Y).
\end{align}
This result was first stated by Shannon \cite{Sha48} with an incomplete proof; the first rigorous proof was provided by Stam \cite{Sta59} using Fisher information and de Bruijn's identity. Various proofs have been found using a number of different approaches, including Young's inequality \cite{Lieb78, DCT91}, restricted Brunn-Minkowski sum \cite{SV00}, minimum mean square estimation (MMSE) \cite{VG06}, and optimal transport \cite{Rio17}. Although the EPI is interesting in its own right, it has further significant impact by connecting several important and active research areas, including functional inequalities on Riemannian manifolds, inequalities in convex geometry, entropic versions of results in additive combinatorics and bounding the capacity of communication channels. We refer to the nice survey by Madiman, Melbourne and Xu \cite{MMX16} for more details.

As a generalization of Shannon entropy, R\'enyi entropy is important in quantum information theory, where it can be used as a measure of entanglement. The R\'enyi entropy is also important in ecology and statistics as an index of diversity. For $p\in(0, 1)\cup (1, \infty)$, the R\'enyi entropy of order $p$ (i.e., $p$-R\'enyi entropy) is defined as
$$
h_p(X)=\frac{1}{1-p}\log\int_{\R^n}f(x)^pdx.
$$
Defining by continuity, $h_1(X)$ corresponds to the classical Shannon entropy. By taking limits, we have
$$
h_0(X)=\log|\supp(f)|,
$$
$$
h_\infty(X)=-\log\|f\|_\infty,
$$
where $\supp(f)$ is the support of $f$, and $|\supp(f)|$ is the volume of $\supp(f)$, and $\|f\|_\infty$ is the essential supremum of $f$. The $p$-R\'enyi entropy power is defined by
$$
N_p(X)=e^{2h_p(X)/n}.
$$
Unlike Shannon entropy power, which captures the variances of real-valued Gaussian random variables, the (operational) meaning of R\'enyi entropy power is not clear for the author. It is observed by Bobkov and Chistyakov \cite{BC15} that Shannon's EPI \eqref{eq:epi} does not always hold for the R\'enyi entropy power. So, a more natural definition of the R\'enyi entropy power is necessary. One variant of the R\'enyi entropy power is $N_p^{\alpha}(X)$ with an appropriate exponent $\alpha$ which depends on $p$. Analogous to the EPI \eqref{eq:epi}, one would like the following R\'enyi EPI to hold for independent random vectors
\begin{align}\label{eq:r-epi}
N_p^{\alpha}(X+Y)\geq N_p^{\alpha}(X)+N_p^{\alpha}(Y).
\end{align}
For instance, the Brunn-Minkowski inequality suggests that $\alpha=1/2$ for $p=0$.
For $p>1$, Bobkov and Marsiglietti \cite{BM16} recently showed that \eqref{eq:r-epi} holds for $\alpha=(p+1)/2$. In the next section, we will show that \eqref{eq:r-epi} holds with a smaller exponent.

Related to the R\'enyi EPI is Savar\'e and Toscani's extension \cite{ST14} of Costa's concavity of entropy power theorem \cite{Cos85} from Shannon entropy to R\'enyi entropy. More specifically, let $u_t$ solve the non-linear heat equation
$
\partial_t u_t=\Delta u_t^p
$
in $\R^n$. For $p>1-2/n$, they showed  that $N_p^\alpha(u_t)$ is a concave function of $t>0$, where $\alpha=1+n(p-1)/2$. Another variant of the R\'enyi EPI obtained by Bobkov and Chistyakov \cite{BC15} states that, at the expense of some
universal factor $c$, for $p>1$ and independent random vectors $X_1, \cdots, X_k$,
$$
N_p(X_1+\cdots+X_k)\geq c\sum_{i=1}^kN_p(X_i).
$$
A sharpened version was obtained by Ram and Sason \cite{RS16}.

In general, the EPI can not be reversed. There exist independent real-valued random variables $X, Y$ with finite entropies, but $h(X+Y)=\infty$, see \cite{BC15} for examples. The reversibility of EPI was addressed by Bobkov and Madiman \cite{BM12}, which states that for independent $\kappa$-concave (the definition is given in Section \ref{sec:conjs}) random vectors $X, Y$ in $\R^n$, there exist linear volume preserving maps $T_1, T_2$ such that
$$
N(T_1(X)+T_2(Y))\leq c\cdot(N(X)+N(Y)),
$$
where $c$ is some constant depending on $\kappa$. This can be thought of as the functional lifting of V. Milman's well known reverse Brunn-Minkowski inequality \cite{Mil86}. The following reverse R\'enyi EPI proposed by Ball, Nayar and Tkocz \cite{BNT15} (Shannon entropy version for log-concave random vectors) and Madiman, Melbourne and Xu \cite{MMX16} states that for any symmetric $\kappa$-conave random vector $(X, Y)$ in $\R^2$,
\begin{align}\label{eq:r-r-epi}
N_p^{1/2}(X+Y)\leq N_p^{1/2}(X)+N_p^{1/2}(Y).
\end{align}
Toward this conjecture, it was shown in \cite{BNT15} that \eqref{eq:r-r-epi} holds for $p=1$ in the log-concave case with an exponent $1/10$ instead of $1/2$. Extension in the R\'enyi entropy setting can be found in \cite{MMX16}. We will verify the conjecture in two cases $p=0$ and $p=2$ for log-concave random vectors.

The paper is organized as follows. In Section \ref{sec:r-epi}, we follow the approach of Dembo, Cover and Thomas \cite{DCT91} of employing Young's inequality to derive an improvement of the R\'enyi EPI recently obtained by Bobkov and Marsiglietti \cite{BM16}. Section \ref{sec:conjs} is devoted to the discussion of various formulations of the conjecture \eqref{eq:r-r-epi}. In particular, we will relate the conjecture to the convexity of the $p$-th cross-section bodies introduced by Gardner and Giannopoulos \cite{GG99}. We verify the conjecture \eqref{eq:r-r-epi} in two cases $p=0, 2$ in Section \ref{sec:special cases}. Connections among various $p$-th mean bodies are discussed in Section \ref{sec:convex bodies}.


\section{R\'enyi EPI} \label{sec:r-epi}

The following statement about the R\'enyi EPI is analogous to Lieb's observation of an equivalent formulation of Shannon's EPI \cite{Lieb78}. For the sake of completeness, we include the proof. 

\begin{lem}\label{lem:linearization}
Let $p\geq0$ and $\alpha>0$ be some constants. Suppose that $X, Y$ are independent random vectors in $\R^n$ such that $X, Y$ and $X+Y$ have finite $p$-R\'enyi entropies. The following statements are equivalent:
\begin{enumerate}
\item
We have
\begin{align}\label{eq:r-epi-1}
N_p^{\alpha}(X+Y)\geq N_p^{\alpha}(X)+N_p^{\alpha}(Y).
\end{align}

\item
For any $\lambda\in[0, 1]$, we have
\begin{align}\label{eq:renyi-epi-linear}
h_p(\lambda^{\frac{1}{2\alpha}}X+(1-\lambda)^{\frac{1}{2\alpha}}Y)\geq \lambda h_p(X)+(1-\lambda)h_p(Y).
\end{align}
\end{enumerate}
\end{lem}

\begin{proof}
To see that \eqref{eq:r-epi-1} implies \eqref{eq:renyi-epi-linear}, it suffices to assume that $\lambda^{\frac{1}{2\alpha}}X+(1-\lambda)^{\frac{1}{2\alpha}}Y$ has finite entropy. By definition,
\begin{align*}
h_p(\lambda^{\frac{1}{2\alpha}}X+(1-\lambda)^{\frac{1}{2\alpha}}Y) &= \frac{n}{2\alpha}\log N_p^\alpha(\lambda^{\frac{1}{2\alpha}}X+(1-\lambda)^{\frac{1}{2\alpha}}Y)\\
&\geq \frac{n}{2\alpha}\log(N_p^\alpha(\lambda^{\frac{1}{2\alpha}}X)+N_p^\alpha((1-\lambda)^{\frac{1}{2\alpha}}Y)
)\\
&= \frac{n}{2\alpha}\log\left(\lambda N_p^\alpha(X)+(1-\lambda)N_p^\alpha(Y)\right)\\
&\geq \frac{n}{2\alpha}\left(\lambda \log N_p^\alpha(X)+(1-\lambda)\log N_p^\alpha(Y)\right)\\
&= \lambda h_p(X)+(1-\lambda)h_p(Y).
\end{align*}
We use \eqref{eq:r-epi-1} in the first inequality. The second inequality follows from the concavity of the $\log$ function. In the second identity, we use the scaling property that $h_p(aX)=h_p(X)+n\log|a|$. For the reverse, i.e. \eqref{eq:renyi-epi-linear} implies \eqref{eq:r-epi-1}, it is not hard to check that
$$
h_p(\lambda^{-\frac{1}{2\alpha}}X)=h_p((1-\lambda)^{-\frac{1}{2\alpha}}Y),
$$
where
$$
\lambda=\frac{N_p^\alpha(X)}{N_p^\alpha(X)+N_p^\alpha(Y)}.
$$
Then, we have
\begin{align*}
h_p(X+Y) &= h_p(\lambda^{\frac{1}{2\alpha}}\cdot\lambda^{-\frac{1}{2\alpha}}X+(1-\lambda)^{\frac{1}{2\alpha}}\cdot(1-\lambda)^{-\frac{1}{2\alpha}}Y)\\
&\geq h_p(\lambda^{-\frac{1}{2\alpha}}X)\\
&= h_p(X)-\frac{n}{2\alpha}\log \lambda\\
&=\frac{n}{2\alpha}\log(N_p^\alpha(X)+N_p^\alpha(Y)).
\end{align*}
This is equivalent to the desired statement.
\end{proof}

\begin{thm}\label{thm:r-epi}
Let $p>1$. Let $X, Y$ be independent random vectors in $\R^n$ such that $X, Y$ and $X+Y$ have finite $p$-R\'enyi entropies. Then, we have
$$
N_p^{\alpha}(X+Y)\geq N_p^{\alpha}(X)+N_p^{\alpha}(Y),
$$
where
\begin{align*}
\alpha=\left(1+\frac{1}{\log2}\left(\frac{p+1}{p-1}\log\frac{p+1}{2p}+\frac{\log p}{p-1}\right)\right)^{-1}.
\end{align*}
\end{thm}

\begin{proof}
Let $0<\lambda<1$. Let $f$ and $g$ be the densities of $\lambda^{\frac{1}{2\alpha}}X$ and $(1-\lambda)^{\frac{1}{2\alpha}}Y$, respectively. Then, $X$ and $Y$ have densities $f_{\lambda, \alpha}(x)=\lambda^{\frac{n}{2\alpha}}f(\lambda^{\frac{1}{2\alpha}}x)$
and $g_{1-\lambda, \alpha}(x)=(1-\lambda)^{\frac{n}{2\alpha}}g((1-\lambda)^{\frac{1}{2\alpha}}x)$, respectively. By Lemma \ref{lem:linearization}, it suffices to prove \eqref{eq:renyi-epi-linear}, which is equivalent to
\begin{align*}
\|f\ast g\|_p^{-p'} &\geq \|f_{\lambda, \alpha}\|_p^{-\lambda p'}\|g_{1-\lambda, \alpha}\|_p^{-(1-\lambda) p'}\\
&= (\lambda^{\lambda}(1-\lambda)^{1-\lambda})^{-\frac{n}{2\alpha}}\|f\|_p^{-\lambda p'}\|g\|_p^{-(1-\lambda)p'},
\end{align*}
where $p'$ is the H\"{o}lder conjugate of $p$, i.e., $1/p+1/p'=1$.
For $p>1$, we have $p'>1$. The above inequality is equivalent to
\begin{align}\label{eq:equi-1}
\|f\ast g\|_p\leq(\lambda^\lambda(1-\lambda)^{1-\lambda})^{\frac{n}{2\alpha p'}}\|f\|_p^{\lambda }\|g\|_p^{1-\lambda}.
\end{align} 
Using the log-convexity of $L^s$ norm, for $1\leq q\leq p$, we have
$$
\|f\|_q\leq \|f\|_1^{\theta_q}\|f\|_p^{1-\theta_q}=\|f\|_p^{1-\theta_q},
$$
where $\theta_q=1-p'/q'$ is the solution of $1/q=\theta_q+(1-\theta_q)/p$. (It also easily follows from H\"{o}lder's inequality). Similarly,  for $1\leq r\leq p$, we have
$$
\|f\|_r\leq\|f\|_p^{1-\theta_r},
$$
where $\theta_r=1-p'/r'$ is the solution of $1/r=\theta_r+(1-\theta_r)/p$.
Then, the estimate \eqref{eq:equi-1} holds if we have
\begin{align}\label{eq:equi-2}
\|f\ast g\|_p\leq(\lambda^\lambda(1-\lambda)^{1-\lambda})^{\frac{n}{2\alpha p'}}\|f\|_q^{\frac{\lambda}{1-\theta_q}}\|g\|_r^{\frac{1-\lambda}{1-\theta_r}}.
\end{align}
Given $p>1$ and $0<\lambda<1$, we can define $q, r$ such that $\lambda=p'/q'$ and $1-\lambda=p'/r'$. It is easy to check that $1+1/p=1/q+1/r$. Young's inequality says
$$
\|f\ast g\|_p\leq \left(\frac{c_qc_r}{c_p}\right)^{n/2}\|f\|_q\|g\|_r,
$$
where $c_s=s^{1/s}(s')^{-1/s'}$. Hence, the estimate \eqref{eq:equi-2} holds if we have
$$
(\lambda^\lambda(1-\lambda)^{1-\lambda})^{\frac{n}{2\alpha p'}}\geq\left(\frac{c_qc_r}{c_p}\right)^{n/2}.
$$
Particularly, we can set
$$
\alpha=\sup_{q, r}\frac{H(\lambda)}{p'\log\left(\frac{c_p}{c_qc_r}\right)},
$$
where $H(\lambda)$ is the Shannon entropy of Bernoulli($\lambda$), and the supremum is taken over all $q, r$ such that $1/q+1/r=1+1/p$. Using $c_s=s^{1/s}(s')^{-1/s'}$, we have
\begin{align*}
p'\log\left(\frac{c_p}{c_qc_r}\right) &= p'\left(\frac{\log p}{p}-\frac{\log q}{q}-\frac{\log r}{r}\right)-\left(\log p'-\frac{p'}{q'}\log q'-\frac{p'}{r'}\log r'\right)\\
&= p'\left(\frac{\log p}{p}-\frac{\log q}{q}-\frac{\log r}{r}\right)-\left(\frac{p'}{q'}\log \frac{p'}{q'}+\frac{p'}{r'}\log \frac{p'}{r'}\right)\\
&= H(\lambda)-A(\lambda),
\end{align*}
where
$$
A(\lambda)=p'\left(\frac{1}{p}\log p-\frac{1}{q}\log q-\frac{1}{r}\log r\right),
$$
which can be rewritten as
$$
p'\left(\left(1-\frac{1}{p'}\right)\log \left(1-\frac{1}{p'}\right)-\left(1-\frac{\lambda}{p'}\right)\log \left(1-\frac{\lambda}{p'}\right)-\left(1-\frac{1-\lambda}{p'}\right)\log \left(1-\frac{1-\lambda}{p'}\right)\right).
$$
Thus, we need to solve the optimization problem
$$
\alpha=\left(1-\sup_{0\leq \lambda\leq 1}\frac{A(\lambda)}{H(\lambda)}\right)^{-1}.
$$
The theorem follows from the fact that $A(\lambda)/H(\lambda)$ achieves the maximum at $\lambda=1/2$.
Notice that $H(\lambda)$ and $A(\lambda)$ are symmetric about $\lambda=1/2$. It suffices to show that $A(\lambda)/H(\lambda)$ is increasing for $0\leq\lambda\leq1/2$. The latter statement is equivalent to
\begin{align}\label{eq:increase1}
\frac{A'(\lambda)}{H'(\lambda)}\geq\frac{A(\lambda)}{H(\lambda)}, ~0\leq\lambda\leq1/2,
\end{align}
where $A'(\lambda)=\log\frac{p'-\lambda}{p'-(1-\lambda)}$, and $H'(\lambda)=\log\frac{1-\lambda}{\lambda}$.
Notice that $A(0)=H(0)=0$. We have
$$
\frac{A(\lambda)}{H(\lambda)}=\frac{\int_0^\lambda A'(s)ds}{\int_0^\lambda H'(s)ds}=\lim_{m\to\infty}\frac{\sum_{k=1}^m A'(k\lambda/m)}{\sum_{k=1}^m H'(k\lambda/m)}.
$$
The statement \eqref{eq:increase1} holds, if we can show that $A'(\lambda)/H'(\lambda)$ is increasing for $0\leq\lambda\leq 1/2$. This is equivalent to
\begin{align}\label{eq:increase2}
\frac{A''(\lambda)}{H''(\lambda)}\leq\frac{A'(\lambda)}{H'(\lambda)}, ~0\leq\lambda\leq1/2,
\end{align}
where $A''(\lambda)=\frac{1-2p'}{(p'-\lambda)(p'-(1-\lambda))}$ and $H''(\lambda)=-\frac{1}{\lambda(1-\lambda)}$. Notice that $A'(1/2)=H'(1/2)=0$. We have
$$
\frac{A'(\lambda)}{H'(\lambda)}=\frac{\int_\lambda^{1/2} A''(s)ds}{\int_\lambda^{1/2} H''(s)ds}=\lim_{m\to\infty}\frac{\sum_{k=0}^{m-1}A''(\lambda+k(1/2-\lambda)/m)}{\sum_{k=0}^{m-1}H''(\lambda+k(1/2-\lambda)/m)}.
$$
The statement \eqref{eq:increase2} holds, if we can show that $A''(\lambda)/H''(\lambda)$ is increasing for $0\leq\lambda\leq 1/2$. One can check that
$$
\frac{A''(\lambda)}{H''(\lambda)}=(2p'-1)\left(1+\frac{p'(p'-1)}{\lambda(1-\lambda)}\right)^{-1},
$$
which is increasing for $0\leq\lambda\leq 1/2$ (recall that $p'>1$). Hence, the theorem follows.
\end{proof}

\begin{rmk}
The exponent $\alpha$ in the theorem is strictly smaller than $(p+1)/2$ obtained in \cite{BM16}. For large $p$, we have $\alpha\approx \frac{p-1}{\log_2 p}$, which, in view of the lower bound $\frac{p-1}{2\log_2\left(\frac{p+1}{2}\right)}$ obtained in  \cite{BM16}, is asymptotically optimal up to a multiplicative constant. As $p\to1$, we can recover the classical EPI.  
Costa, Hero and Vignat \cite{CHV03} proved that, under covariance constraints,  maximizers of R\'enyi entropies are generalized Gaussian distributions, indexed by a parameter $-\infty<\beta\leq \frac{2}{n+1}$, of the form 
$$
A_\beta\left(1-\frac{\beta}{2}\|x\|^2\right)_+^{\frac{1}{\beta}-\frac{n}{2}-1}.
$$
It is not clear whether the sharp exponent can be obtained from the generalized Gaussian case. As $p\to0$, we have $\alpha\to 1/2$, which captures the classical Brunn-Minkowski inequality. But it is not clear if Theorem \ref{thm:r-epi} holds for $0<p<1$, in which case we might need a reverse H\"{o}lder argument.
\end{rmk}


\section{Reverse R\'enyi EPI} \label{sec:r-r-epi}

This section is devoted to the reversibility of the R\'enyi EPI. In general, the R\'enyi EPI can not be reversed, even in the Shannon case. Variants of the reverse R\'enyi EPI might hold for certain distributions.

\subsection{Discussion of conjectures} \label{sec:conjs}
Recall that a probability measure $\mu$ on $\R^n$ is called $\kappa$-concave if the following Brunn-Minkowski type inequality
$$
\mu(\lambda A+(1-\lambda)B)\geq (\lambda \mu(A)^\kappa+(1-\lambda)\mu(B)^\kappa)^{1/\kappa}
$$
holds for any $0\leq \lambda\leq 1$ and any non-empty Boreal sets $A, B\subseteq\R^n$. Defining by continuity, for $\kappa=0$, we have 
$$
\log\mu(\lambda A+(1-\lambda)B)\geq \lambda \log\mu(A)+(1-\lambda)\log\mu(B),
$$
which characterizes the class of log-concave measures. For $\kappa=-\infty$, we obtain the largest class, whose members are called convex measures. A $\kappa$-concave measure is absolutely continuous with respect to the Lebesgue measure if $\kappa\leq 1/n$, and its density $f$ satisfies the so-called $s$-concave property
$$
f(\lambda x+(1-\lambda)y)\geq(\lambda f(x)^s+(1-\lambda)f(y)^s)^{1/s}
$$
for all $x, y\in\R^n$ such that $f(x)f(y)\neq0$ and $s=\kappa/(1-n\kappa)$. We refer to the seminal work of Borell \cite{Bor75} for the hierarchical properties of such measures. 

The following conjecture was raised by Ball, Nayar and Tkocz \cite{BNT15} (Shannon entropy version for log-concave measures) and Madiman, Melbourne and Xu \cite{MMX16} (R\'enyi entropy version for general $\kappa$-concave measures). 

\begin{conj}\label{conj:1}
Let $p\in[0, \infty]$. Let $\kappa\in[-\infty, 1/n]$, and let $X$ be a $\kappa$-concave random vector in $\R^n$.  The function 
$$
v\mapsto
\begin{cases}
N^{1/2}_p(v\cdot X), & v\neq0\\
0, & v=0
\end{cases}
$$
defines a norm of $v\in\R^n$.
\end{conj}

It is easy to check the homogeneity. The veracity of the conjecture really depends on proving the triangle inequality
$$
N^{1/2}_p((u+v)\cdot X)\leq N^{1/2}_p(u\cdot X)+N^{1/2}_p(v\cdot X),
$$
which essentially depends on the marginal distribution of $X$ on the plane spanned by $u, v$. We can apply a linear transformation to the plane such that $u, v$ end up perpendicular without changing the convexity of the marginal distribution. Hence, the conjecture is equivalent to that 
\begin{align}\label{eq:r-r-epi-1}
N_p^{1/2}(X+Y)\leq N_p^{1/2}(X)+N_p^{1/2}(Y)
\end{align}
holds for any symmetric $\kappa$-concave random vector $(X, Y)$ in $\R^2$. Similar to Lemma \ref{lem:linearization}, we have the following linearized version of \eqref{eq:r-r-epi-1}.  Let $(X, Y)$ be a symmetric $\kappa$-concave random vector in $\R^2$, and assume that $X, Y$ have equal $p$-R\'enyi entropy. For any $\lambda\in[0, 1]$, we have
\begin{align}\label{eq:r-epi-linear}
h_p(\lambda X+(1-\lambda)Y)\leq h_p(X).
\end{align}
It is not hard to check that equality holds for $p=0$. 
The statement is not true without the equal entropy assumption. Dembo, Cover and Thomas \cite{DCT91} proved that for any independent random vectors $X$ and $Y$ in $\R^n$,
$$
h_0(\lambda X+(1-\lambda)Y)\geq \lambda h_0(X)+(1-\lambda)h_0(Y).
$$
As is observed in \cite{BNT15}, the statement \eqref{eq:r-r-epi-1} holds for $p=1$ in the Shannon case when $(X, Y)$ is log-concave and one marginal has the same law as the other one rescaled, say $Y\sim tX$ for some $t>0$.  The essential case $t=1$ was first observed by Cover and Zhang \cite{CZ94}. 

An interesting related problem one would like to study is the maximal entropy increment $h(X+Y)-h(X)$, where $Y$ is an independent copy of $X$. Abbe \cite{Abb} conjectured that exponential distribution is the maximizer among log-concave measures for given covariance matrix. This might be related to the minimal entropy problem for log-concave measures with fixed covariance matrix. The latter is the entropy version of Bourgain's slicing problem \cite{Bou86} formulated by Bobkov and Madiman \cite{BM11-1}. The lower bound of the entropy jump $h(X+Y)-h(X)$ 
has been studied in \cite{BBN03, ABBN04, BN12} under the assumption that $X$ satisfies the Poincar\'e inequality.

The following statement is stronger than Conjecture \ref{conj:1}.
\begin{conj}
Let $(X, Y)$ be a symmetric $\kappa$-concave random vector in $\R^2$, and assume that $X, Y$ have equal $p$-R\'enyi entropy. The function
$$
\lambda\mapsto h_p(\lambda X+(1-\lambda)Y)
$$
is convex for $\lambda\in[0, 1]$.
\end{conj}

Before giving another geometric formulation of Conjecture \ref{conj:1}, we introduce some notations commonly used in convex geometry. Let $K\subset\R^n$ be a compact set with the origin in its interior. For $v\in\S^{n-1}$, the radial function $\rho_K(v)$ is defined as
$$
\rho_K(v)=\sup\{r\geq0: rv\in K\}.
$$
For a symmetric convex body $K$, it is not hard to see that $\rho_K(v)=\|v\|_{K}^{-1}$, where $\|v\|_{K}$ is the norm of $v$ for which $K$ is the unit ball. 

Conjecture \ref{conj:1} is true when $p=\infty$ and $X$ is uniformly distributed over a symmetric convex body $K$. This is known as Buseman's theorem \cite{Bus49}, 
and the unit ball is called the intersection body, $I(K)$, of $K$. 
Generalizations of Buseman's theorem hold for symmetric log-concave measures \cite{Ball88} and more general convex measures \cite{MMX16}.

As a generalization of intersection body, the $p$-th cross-section body of a convex body was introduced by Gardner and Giannopoulos \cite{GG99}. Its extension for probability measures is given below. We will see that that the $p$-th cross-section body of a random vector (its density) agrees with the unit ball associated to the norm claimed in Conjecture \ref{conj:1}. 

\begin{defn}
Let $f$ be a probability density on $\R^n$. The $p$-th cross-section body $C_p(f)$ is the set with the radial function
\begin{align}\label{eq:cpf-radial}
\rho_{C_p(f)}(v)=\left(\int_{\R^n}\left(\int_{x+v^\perp}f(y)dy\right)^pf(x)dx\right)^{1/p},~v\in\S^{n-1},
\end{align}
where $v^\perp$ is the subspace perpendicular to $v$. 
\end{defn}

It is not hard to see that the radial function can be rewritten as
$$
\rho_{C_pf}(v)=\left(\int_{\R}\left(\int_{x\cdot v=t}f(x)dx\right)^{p+1}dt\right)^{1/p}.
$$
For $v\in\R^n$, the density of $v\cdot X$ is
$$
|v|^{-1}\int_{v\cdot x=t}f(x)dx.
$$
Therefore, we have
$$
\rho_{C_pf}(v)=N^{-1/2}_{p+1}(v\cdot X).
$$
Hence, we can rephrase Conjecture \ref{conj:1} in the following way.
\begin{conj}\label{conj:2}
Let $p\in[-1, \infty]$. Let $s\in[-1/n, \infty]$, and let $f$ be a symmetric $s$-concave probability density on $\R^n$. Then, the $p$-th corss-section body $C_p(f)$ is a symmetric convex body.
\end{conj}

We introduce other types of $p$-th mean bodies, which are closely related to each other. Their relationships will be discussed in Section \ref{sec:convex bodies}. The following is a natural extension of the intersection body of a convex body to a probability density $f$.

\begin{defn}
Let $f$ be a probability density on $\R^n$. The intersection body $I(f)$ is the set with the radial function
\begin{align*}
\rho_{I(f)}(v)=\int_{v^\perp}f(x)dx,~v\in\S^{n-1}.
\end{align*}
\end{defn}

The following $p$-th radial mean body of a probability density function reduces to the $p$-th radial mean body of a convex body introduced by Gardner and Zhang \cite{GZ98} with some normalizing constant if the density is taken to be uniform over that convex body.  
\begin{defn}
Let $f$ be a probability density on $\R^n$. The $p$-th radial mean body $R_p(f)$ is the set with the radial function
\begin{align}\label{eq:rpf-radial}
\rho_{R_p(f)}(v)=\left(\int_{\R^n}f(x)\int_{\R_+}r^{p-1}f(x+rv)drdx\right)^{1/p},~v\in\S^{n-1}.
\end{align}
\end{defn}

\begin{defn}
Let $f$ be a probability density function on $\R^n$. Ball's $p$-th mean body is defined as the set with radial function
\begin{align*}
\rho_{B_p(f)}(v)=\left(\int_{\R_+}r^{p-1}f(rv)dr\right)^{1/p}, ~v\in\S^{n-1}.
\end{align*}
\end{defn}

\begin{defn}
Let $f$ be a probability density on $\R^n$. The polar $p$-th centroid body $\Gamma_p^\circ (f)$ is the set with the radial function
\begin{align*}
\rho_{\Gamma_p^\circ (f)}(v)=\left(\int_{\R^n}|v\cdot x|^pf(x)dx\right)^{-1/p},~v\in\S^{n-1}.
\end{align*}
where $v\cdot x$ is the inner product of $v$ and $x$.
\end{defn}

We let $Z_p(f)$ be the dilation of $\Gamma_p^\circ (f)$ defined as
\begin{align}\label{eq:dilated}
Z_p(f)=\left(\frac{2}{p+1}\right)^{1/p}\Gamma_p^\circ (f).
\end{align}

\subsection{$C_{-1}(f)$ and $C_1(f)$}\label{sec:special cases}

In the following, we verify Conjecture \ref{conj:2} in two special cases $p=-1, 1$ (log-concave densities when $p=1$), equivalently Conjecture \ref{conj:1} for $p=0, 2$. As we mentioned, it is easy to see that \eqref{eq:r-epi-linear} holds for $p=0$. Equivalently, Conjecture \ref{conj:1} holds for $p=0$, and Conjecture \ref{conj:2} holds for $p=-1$. We include this as a theorem to show the relation between $C_{-1}(f)$ and other geometric bodies.

\begin{thm}
Let $f$ be a probability density supported on a symmetric convex body $K\subset\R^n$. We have
$$
C_{-1}(f)=(2K)^\circ=(R_\infty (f))^\circ.
$$
\end{thm}

\begin{proof}
In this case, we have
$$
N_0^{1/2}(v\cdot X)=|\text{Range}(v\cdot X)|=2\sup_{x\in K}v\cdot x=2h_K(v)=h_{2K}(v).
$$
Since $K$ is a symmetric convex body, the dilated set $2K=\{2x: x\in K\}=K\pm K$ is a symmetric convex body as well. Therefore, the support function $h_{2K}(v)$ is a homogeneous convex function, and it defines a norm on $\R^n$. Then, we have
$$
\rho_{C_{-1}(f)}(v)=h_{2K}(v)^{-1},~v\in\S^{n-1},
$$
which implies that
$$
C_{-1}(f)=(2K)^\circ.
$$
To see the 2nd identity in the theorem, we let $p\to\infty$ in the definition \eqref{eq:rpf-radial}. We have
\begin{align*}
\rho_{R_\infty (f)}(v) &= \sup\{r>0: \text{there is $x\in K$ such that}~ x+rv\in K\}\\
&= \sup\{r>0: rv\in K-K\}\\
&= \rho_{2K}(v).
\end{align*}
The first identity follows from the principle of
the largest term. This is equivalent to the desired statement.
\end{proof}

\begin{rmk}
If we allow the norm $\|v\|=N_0^{1/2}(v\cdot X)$ to be infinite, the support of $f$ only need to be symmetric and convex, not necessary bounded.
\end{rmk}

\begin{thm}[Ball's Theorem \cite{Ball88}]
Let $p>0$ and let $f$ be a symmetric log-concave probability density on $\R^n$. Then 
$$
\rho(v)=\left(\int_{\R_+}r^{p-1}f(rv)dr\right)^{1/p},~v\in\S^{n-1}
$$
is the radial function of a symmetric convex body in $\R^n$.
\end{thm}

\begin{thm}\label{thm:c1f}
Let $p>0$, and let $f$ be a log-concave (not necessary symmetric) probability density on $\R^n$. Then, $C_1(f)$ and $R_p(f)$ are symmetric convex bodies. Furthermore, we have
$$
C_1(f)=I(\hat{f})=(n-1)I(R_{n-1}(f)),
$$
$$
R_p(f)=B_p({\hat{f}}),
$$
where
$$
\hat{f}(x)=\int_{\R^n}f(y)f(x+y)dy.
$$
\end{thm}

\begin{proof}
By the definition of $C_1(f)$ in \eqref{eq:cpf-radial}, for $v\in\S^{n-1}$, we have
\begin{align*}
\rho_{C_1(f)}(v) &= \int_{\R^n}f(x)\int_{x+v^\perp}f(y)dydx\\
&= \int_{\R^n}f(x)\int_{v^\perp}f(x+y)dydx\\
&=\int_{v^\perp}\int_{\R^n}f(x)f(x+y)dxdy\\
&= \rho_{I(\hat{f})}(v).
\end{align*} 
To see the second identity, we rewrite $\rho_{C_1(f)}(v)$ in the spherical polar coordinates as
\begin{align*}
\rho_{C_1(f)}(v) &= \int_{\R^n}f(x)\int_{\S^{n-1}\cap v^\perp}\int_{\R_+}r^{n-2}f(x+ru)drdudx\\
&=\int_{\S^{n-1}\cap v^\perp}\int_{\R^n}f(x)\int_{\R_+}r^{n-2}f(x+ru)drdxdu.
\end{align*}
The 2nd equation follows from Fubini's theorem. Using the definition of $R_{n-1}(f)$ in \eqref{eq:rpf-radial}, we have
\begin{align*}
\rho_{C_1(f)}(v) &= \int_{\S^{n-1}\cap v^\perp}\rho_{R_{n-1}(f)}(u)^{n-1}du\\
&= (n-1)\int_{\S^{n-1}\cap v^\perp}\int_0^{\rho_{R_{n-1}(f)}(u)}r^{n-2}drdu\\
&= (n-1)|R_{n-1}(f)\cap v^\perp|\\
&= (n-1)\rho_{I(R_{n-1}(f))}(v).
\end{align*}
By the definition of $R_p(f)$ in \eqref{eq:rpf-radial}, we have
\begin{align*}
\rho_{R_p(f)}(v) &= \left(\int_{\R^n}f(x)\int_{\R_+}r^{p-1}f(x+rv)drdx\right)^{1/p}\\
&= \left(\int_{\R_+}r^{p-1}\int_{\R^n}f(x)f(x+rv)dxdr\right)^{1/p}\\
&= \left(\int_{\R_+}r^{p-1}\hat{f}(rv)dr\right)^{1/p}\\
&= \rho_{B_p(\hat{f})}(v).
\end{align*}
Notice that $\hat{f}$ is a symmetric log-concave function, since it is the density of the difference of two independent random vectors with the same density $f$. Ball's theorem implies the convexity of $B_p(\hat{f})$, i.e., $R_p(f)$, and $C_1(f)$, i.e., $I(\hat{f})$.
\end{proof}

\begin{rmk}
The theorem can be thought of as the functional version of Theorem 5.2 in \cite{GG99} and Theorem 4.3 in \cite{GZ98}. The proofs given here are much simpler. Generalization of Ball's theorem do hold for convex measures \cite{Bob10}. But it is not clear for us if the proof can be extended to general convex measures, since the convolution of two $s$-concave densities is not necessary $s'$-concave for some $s'$.
\end{rmk}

\section{Discussion of $p$-th mean bodies} \label{sec:convex bodies}

The various $p$-th mean bodies introduced in the previous section are closely related to each other. We take $f=1_K$, where $K$ is a convex body. It is known that $R_\infty(K)$ is the difference body of $K$, and that $R_{-1}(K)$ is the polar projection body of $K$. So, the $p$-th radial mean body, $R_p(K)$, forms a spectrum connecting the difference body and polar projection body. Similar property holds for the dilated polar $p$-th centroid body defined in \eqref{eq:dilated}, $Z_p(K)$, which connects the intersection body of $K$ when $p=-1$ and the polar body of $K$ when $p=\infty$.

The next two propositions are functional liftings of Lemma 5.1 and Proposition 3.1 in \cite{GG99}, which could be used to give an alternative proof of the first part of Theorem \ref{thm:c1f}.

\begin{prop}\label{prop:Z-R}
Let $p\geq-1$ be non-zero. Let $f$ be a probability density on $\R^n$. For $v\in\S^{n-1}$, we have
$$
\rho_{Z_p(R_{n+p}(f))}(v)=\left(\frac{1}{n+p}\int_{\R^n}f(x)\rho_{Z_p(f_x)}(v)^{-p}dx\right)^{-1/p},
$$
where $f_x(y)=f(x+y)$.
\end{prop}

\begin{proof}
\begin{align*}
\rho_{Z_p(R_{n+p}(f))}(v)^{-p} &= \frac{p+1}{2}\int_{R_{n+p}(f)}|v\cdot x|^pdx\\
&=\frac{p+1}{2}\int_{\S^{n-1}}|u\cdot v|^p\int_0^{\rho_{R_{n+p}(f)}(u)}r^{n+p-1}drdu\\
&=\frac{p+1}{2(n+p)}\int_{\S^{n-1}}|u\cdot v|^p\rho_{R_{n+p}(f)}(u)^{n+p}du\\
&= \frac{p+1}{2(n+p)}\int_{\S^{n-1}}|u\cdot v|^p\int_{\R^n}f(x)\int_{\R_+}r^{n+p-1}f_x(ru)drdxdu\\
&= \frac{p+1}{2(n+p)}\int_{\R^n}f(x)\int_{\S^{n-1}}|u\cdot v|^p\int_{\R_+}r^{n+p-1}f_x(ru)drdudx\\
&= \frac{p+1}{2(n+p)}\int_{\R^n}\left(\int_{\R^n}|v\cdot y|^pf_x(y)dy\right)f(x)dx\\
&=\frac{1}{n+p}\int_{\R^n}f(x)\rho_{Z_p(f_x)}(v)^{-p}dx. 
\end{align*}
\end{proof}

For $p>-1$, the $p$-cosine transform $T_pf$ of a continuous function $f$ on $\S^{n-1}$ is defined as
$$
T_pf(v)=\int_{\S^{n-1}}|u\cdot v|^pf(u)du,~ v\in\S^{n-1}.
$$
The spherical Radon transform $\mathcal{R}f$ is defined as
$$
\mathcal{R}f(v)=\int_{\S^{n-1}\cap v^\perp}f(u)du,~ v\in\S^{n-1}.
$$
The following connection between $p$-cosine transform and Radon transform is known, see \cite{Kol97, GZ99}.
\begin{align}\label{eq:T-R}
\lim_{p\to-1+}\frac{p+1}{2}T_pf=\mathcal{R}f.
\end{align}
\begin{prop} \label{prop:Z-I}
Let $f$ be a probability density on $\R^n$. For $v\in\S^{n-1}$, we have
$$
\lim_{p\to-1+}\rho_{Z_p(f)}(v)^{-p}=\rho_{I(f)}(v).
$$
In other words, we have
$$
Z_{-1}(f)=I(f).
$$
\end{prop}

\begin{proof}
By definition, we have
\begin{align*}
\rho_{Z_p(f)}(v)^{-p} &= \frac{p+1}{2}\int_{\R^n}|v\cdot x|^pf(x)dx\\
&=\frac{p+1}{2}\int_{\S^{n-1}}|u\cdot v|^p\int_{\R_+}r^{n+p-1}f(ru)drdu.\\
\end{align*}
Using identity \eqref{eq:T-R}, we have
\begin{align*}
\lim_{p\to-1+}\rho_{Z_p(f)}(v)^{-p} &= \int_{\S^{n-1}\cap v^\perp}\int_{\R_+}r^{n-2}f(ru)drdu\\
&=\int_{v^\perp}f(x)dx\\
&=\rho_{I(f)}(v).
\end{align*}
\end{proof}

In the following, we apply these properties to give another proof of the first statement of Theorem \ref{thm:c1f}.
\begin{proof}
\begin{align*}
\rho_{C_1(f)}(v) &= \int_{\R^n}f(x)\int_{x+v^\perp}f(y)dydx\\
&=\int_{\R^n}f(x)\int_{v^\perp}f(x+y)dydx\\
&=\int_{\R^n}f(x)\rho_{I(f_x)}(v)dx\\
&=\int_{\R^n}f(x)\rho_{Z_{-1}(f_x)}(v)dx\\
&=(n-1)\rho_{Z_{-1}(R_{n-1}(f))}(v)\\
&=(n-1)\rho_{I(R_{n-1}(f))}(v).
\end{align*}
We use Proposition \ref{prop:Z-I} in the last equation and the 3rd last equation. The 2nd last identity follows from Proposition \ref{prop:Z-R}.
\end{proof}
Using a similar argument, we can prove the following result.
\begin{prop}
Let $f$ be a probability density on $\R^n$. For $v\in\S^{n-1}$,  we have
$$
\lim_{p\to-1+}\frac{p+1}{2(n-1)}T_p\rho_{C_{n-1}(f)}^{n-1}(v)=\rho_{I(C_{n-1}(f))}(v).
$$
\end{prop}

\begin{proof}
For $u\in\S^{n-1}$, we have
\begin{align*}
\rho_{I(C_{n-1}(f))}(v) &=|C_{n-1}f\cap v^\perp| \nonumber\\
&=\int_{C_{n-1}(f)\cap v^\perp}1dx \nonumber\\
&=\int_{\S^{n-1}\cap v^\perp}\int_0^{\rho_{C_{n-1}f}(u)}r^{n-2}drdu \nonumber\\
&=\frac{1}{n-1}\int_{\S^{n-1}\cap v^\perp}\rho_{C_{n-1}f}(u)^{n-1}du \nonumber\\
&=\frac{1}{n-1}\int_{\S^{n-1}\cap v^\perp}\int_{\R^n}f(x)\left(\int_{x+u^\perp}f(y)dy\right)^{n-1}dxdu \nonumber\\
&=\frac{1}{n-1}\int_{\R^n}f(x)\int_{\S^{n-1}\cap v^\perp}\left(\int_{x+u^\perp}f(y)dy\right)^{n-1}dudx.
\end{align*}
Using the connection \eqref{eq:T-R}, we have
\begin{eqnarray*}
\rho_{I(C_{n-1}(f))}(v) &=& \lim_{p\to-1+}\frac{p+1}{2(n-1)}\int_{\R^n}f(x)\int_{\S^{n-1}}|u\cdot v|^p\left(\int_{x+u^\perp}f(y)dy\right)^{n-1}dudx\\
&=& \lim_{p\to-1+}\frac{p+1}{2(n-1)}\int_{\S^{n-1}}|u\cdot v|^p\int_{\R^{n}}f(x)\left(\int_{x+u^\perp}f(y)dy\right)^{n-1}dxdu\\
&=& \lim_{p\to-1+}\frac{p+1}{2(n-1)}\int_{\S^{n-1}}|u\cdot v|^p\rho_{C_{n-1}(f)}^{n-1}(u)du\\
&=&\lim_{p\to-1+}\frac{p+1}{2(n-1)}T_p\rho_{C_{n-1}(f)}^{n-1}(v).
\end{eqnarray*}

\end{proof}

\section*{Acknowledgement}
The author is grateful to Mokshay Madiman and Muriel M\'edard for helpful discussions. The author also would like to thank the anonymous referee for a thorough reading of the paper, and valuable suggestions which make the paper much improved.


\end{document}